\documentclass[12pt, reqno]{amsart}
\usepackage{amsmath, amsthm, amscd, amsfonts, amssymb, graphicx, color}
\usepackage[bookmarksnumbered, colorlinks, plainpages]{hyperref}

\newtheorem{theorem}{Theorem}[section]
\newtheorem{lemma}[theorem]{Lemma}

\theoremstyle{definition}
\newtheorem{definition}[theorem]{Definition}
\newtheorem{example}[theorem]{Example}

\theoremstyle{remark}
\newtheorem{remark}[theorem]{Remark}
\numberwithin{equation}{section}

\def\CC{\mathbb{C}}
\def\DD{\mathfrak{D}}
\def\CX{C(X)}
\def\HCX{\mathcal{H}\bigl(\CX\bigr)}
\def\series{\sum_{n=0}^{\infty}}
\newcommand{\norm}[1]{\left\|#1\right\|}

\begin{document}
\setcounter{page}{1}


\title[Solubility of transcendental equations]{On the solubility of transcendental equations in commutative $C^*$-algebras}

\author[M. Garc\' ia Armas, C. S\' anchez Fern\' andez]{Mario Garc\' ia Armas$^1$ and Carlos S\' anchez Fern\' andez$^2$}

\address{$^{1}$ Facultad de Matem\' atica y Computaci\' on, Universidad de La Habana, Havana, Cuba.}

\email{\textcolor[rgb]{0.00,0.00,0.84}{marioga@matcom.uh.cu}}

\address{$^{2}$ Facultad de Matem\' atica y Computaci\' on, Universidad de La Habana, Havana, Cuba}
\email{\textcolor[rgb]{0.00,0.00,0.84}{csanchez@matcom.uh.cu}}

\dedicatory{}




\begin{abstract}
It is known that $\CX$ is algebraically closed if $X$ is a locally connected, hereditarily unicoherent compact Hausdorff space. For such spaces, we prove that if $F:\CX \to \CX$ is an entire function in the sense of Lorch, i.e., is given by an everywhere convergent power series with coefficients in $\CX$, and satisfies certain restrictions, then it has a root in $\CX$. Our results generalizes the monic algebraic case. 
\end{abstract} \maketitle

\section{Introduction}

Let $X$ be a compact Hausdorff space and let $\CX$ be the Banach algebra of complex-valued continuous functions on $X$. We say that $F : \CX \to \CX$ is \emph{entire} (in the sense of Lorch) if it is Fr\'{e}chet differentiable at every point $w \in \CX$ and its differential is given by a multiplication operator $L_w(h) = F'(w) h$, for some $F'(w) \in \CX$ (see \cite{Lo1943} for details). We denote the set of entire functions by $\HCX$ and make it into a unital algebra with the usual operations. It is well known that $F \in \HCX$ if and only if it admits a power series expansion
\begin{equation}
F(w) = \series a_n w^n, \qquad w \in \CX,
\end{equation}
where $a_n \in \CX$ for all $n \geq 0$, $\limsup_n \norm{a_n}^{1/n} = 0$ and the series converges in norm for each fixed $w\in\CX$. 

To any entire function $F$, we may associate the map $X \times \CC \to \CC$ defined by
\begin{equation} \label{Associate_Function}
(x,z) \mapsto \series a_n(x) z^n\ \Bigl(= F\left(z 1_{\CX}\right)(x)\Bigr),
\end{equation}
which is easily seen to be continuous on $X \times \CC$ and holomorphic with respect to $z$ for $x \in X$ fixed. On the other hand, it is obvious that the above map uniquely determines $F$. By a customary abuse of notation, we also write $F$ for the map in \eqref{Associate_Function}; it should be clear from the context which case we are referring to. 

We say that $F \in \HCX$ has a \emph{root} in $\CX$, if there exists $w \in \CX$ such that $F(x,w(x)) = 0$ for all $x \in X$. If $X$ is a locally connected compact Hausdorff space, it was observed by Miura and Niijima \cite{MN2003} that $\CX$ is algebraically closed, i.e., every monic polynomial with coefficients in the algebra has at least one root in the algebra, if and only if $X$ is hereditarily unicoherent (see also Honma and Miura \cite{HM2007}). We recall that $X$ is said to be hereditarily unicoherent, if the intersection $A \cap B$ is connected for all closed connected subsets $A, B$ of $X$. A short, but accurate introduction to the state of the art in monic algebraic equations can be found in Kawamura and Miura \cite{KM2009}.

However, if we consider more general functions in $\HCX$, the existence of continuous roots is no longer guaranteed, even if $X$ is as simple as the unit interval. For example, the function $F(x,z) = x^2 z - x$ does not have a root in $C([0,1])$. We now introduce two phenomena that arise in the preceding example and have a strong relation with the existence of solutions of the equation $F(w)=0$.

\begin{definition} \label{Def_Degeneracy}
Let $X$ be a compact Hausdorff space. A function $F \in \HCX$ is said to be \emph{degenerate} at $x_0 \in X$ if the map $z \mapsto F(x_0,z)$ is constant; otherwise, it is said to be \emph{nondegenerate} at $x_0$.
\end{definition}

\begin{definition} \label{Def_AsymptoticZero}
Let $X$ be a compact Hausdorff space, let $Y \subset X$ be a connected subset and $x_0 \in \overline{Y}\setminus Y$. A function $w \in C(Y)$ is said to be an \emph{asymptotic root} of $F \in \HCX$ if $F(x,w(x)) = 0$ for all $x \in Y$ and
\begin{equation}
\lim_{x \to x_0} w(x) = \infty, \quad x \in Y.
\end{equation}
\end{definition}

The aim of this paper is to prove that if $X$ is a connected, locally connected, hereditarily unicoherent compact Hausdorff space, then any nowhere degenerate function $F \in \HCX$ with no asymptotic roots, satisfying $F(x_0,z_0) = 0$, has at least one root $w \in \CX$ such that $w(x_0) = z_0$. It is easily seen that monic polynomials are nondegenerate at every point of $X$ and do not have asymptotic roots. Consequently, our result generalizes that of Miura and Niijima \cite{MN2003}.

It is important to mention that Gorin and S\'{a}nchez Fern\'{a}ndez \cite{GS1977} studied the case where $X$ is a connected, locally connected, hereditarily unicoherent, compact metric space and showed that any nowhere degenerate function $F \in \HCX$ with no asymptotic arcs, satisfying the condition $F(x_0,z_0)=0$, has at least one root $w\in C(X)$ such that $w(x_0)=z_0$ (for a definition of asymptotic arc, see \cite{GS1977}). In our work, we do not assume that $X$ is a first-countable space. 

\section{Existence of Roots}

We start by pointing out a very useful lemma, which arises naturally from Rouch\'{e}'s Theorem.

\begin{lemma} \label{Lemma:Algebrization}
Let $X$ be a compact Hausdorff space, $F \in \HCX$ and pick $x_0 \in X$ such that the map $z \mapsto F(x_0,z)$ has a zero $z_0$ of multiplicity $n$. Then, there exist an open disk $D_r(z_0)$ and a neighborhood $V$ of $x_0$ such that
\begin{equation}
F(x,z) = P(x,z)\:G(x,z), \qquad (x,z) \in V \times D_r(z_0),
\end{equation}
where $P(x,z) = z^n + a_1(x)z^{n-1}+\ldots+a_n(x)$ is a monic polynomial with coefficients in $C(V)$ satisfying $P(x_0,z) = (z-z_0)^n$ and $G$ never vanishes in $V \times D_r(z_0)$.
\end{lemma}

\begin{proof}
Set $r>0$ such that the map $z \mapsto F(x_0,z)$ has no roots in $\overline{D_r(z_0)}\setminus \{z_0\}$ and write $\Gamma = \{z \in \CC\: :\:|z-z_0|=r\}$. Also, write $m = \min_{\Gamma} |F(x_0,z)| > 0$. By a standard compactness argument, we can find a neighborhood $V$ of $x_0$ such that $|F(x,z) - F(x_0,z)|<m$ for all $x \in V$ and $z \in \Gamma$. Then, an application of Rouch\'{e}'s Theorem shows that $z \mapsto F(x,z)$ has exactly $n$ zeros in $D_r(z_0)$, counting multiplicities, whenever $x \in V$.

For any $x \in V$, we denote the zeros of $z \mapsto F(x,z)$ in $D_r(z_0)$ by $z_1(x), \ldots, z_n(x)$, taken in any order and we define
\begin{equation}
P(x,z) = \bigl(z-z_1(x)\bigr)\ldots\bigl(z-z_n(x)\bigr) = z^n + a_1(x)z^{n-1}+\ldots+a_n(x).
\end{equation}
Obviously, we have $P(x_0,z) = (z-z_0)^n$. Now, consider the central symmetric functions
\begin{equation}
s_k(x) = \sum_{i=1}^{n} \bigl(z_i(x)\bigr)^k, \qquad k \geq 0.
\end{equation}
Since $z_1(x), \ldots, z_n(x)$ are the zeros of $z \mapsto F(x,z)$ in the interior of $\Gamma$, it is well known (and easily verified) that
\begin{equation}
s_k(x) = \frac{1}{2\pi i} \int_{\Gamma}\:z^k\: \frac{\frac{\partial F}{\partial z} (x,z)}{F(x,z)}\: dz.
\end{equation}
Consequently, $s_k \in C(V)$ for all $k \geq 0$. It is also well known that the functions $s_k$ are connected to the functions $a_k$ via the so-called Newton identities. Therefore, the continuity of $a_k$ for $1 \leq k \leq n$ can be established by an easy induction.

Finally, for $(x,z) \in  V \times D_r(z_0)$, define $G(x,z)$ as the quotient $F(x,z)/P(x,z)$ if $P(x,z)\neq 0$ and set $G(x,z) = 1$ otherwise.
\end{proof}

Before going any further, we need some topological remarks. A good exposition of such facts can be found in \cite{MN2003}, a great deal of which we reproduce for completeness. Let $X$ be a connected topological space. A point $p \in X$ separates the distinct points $a, b \in X\setminus\{p\}$ if there exist disjoint open sets $A$ and $B$ such that $a \in A$, $b \in B$ and $X\setminus\{p\} = A \cup B$. If the point $p$ belongs to every connected closed subset of $X$ containing $a$ and $b$, we say that $p$ cuts $X$ between $a$ and $b$. If $X$ is a locally connected and connected compact Hausdorff space, then $p$ cuts $X$ between $a$ and $b$ if and only if $p$ separates the points $a$ and $b$ (cf. \cite[Theorem 3-6]{HY1988}).

If $X$ is a connected compact Hausdorff space, there exists a minimal connected closed subset, with respect to set inclusion, containing both $a$ and $b$ (cf. \cite[Theorem 2-10]{HY1988}). If $X$ is hereditarily unicoherent, such a minimal set is unique and we denote it by $E[a,b]$. Clearly, every point in $E[a,b]\setminus\{a,b\}$ cuts $X$ between $a$ and $b$. Therefore, if we assume that $X$ is also locally connected, such points also separate $a$ and $b$. We define the separation order $\preceq$ in $E[a,b]$ the following way: for distinct points $p, q \in E[a,b]$, we say that $p \prec q$ if $p = a$ or $p$ separates $a$ and $q$. Then, we write $p \preceq q$ if $p=q$ or $p \prec q$. Such choice makes $E[a,b]$ into a totally ordered space (cf. \cite[Theorem 2-21]{HY1988}). If we define the order topology in $E[a,b]$ the usual way, then it coincides with the induced topology in $E[a,b]$ (cf. \cite[Theorem 2-25]{HY1988}). Also, by \cite[Theorem 2-26]{HY1988}, every non-empty subset of $E[a,b]$ has a least upper bound, i.e., $E[a,b]$ is order-complete.

To avoid repetitions, we assume henceforth that $X$ is a connected, locally connected, hereditarily unicoherent compact Hausdorff space, unless stated otherwise. 

\begin{lemma}\label{Lemma:E[a,b]} The following two properties hold:

i-) Any connected subset of $X$ containing $a$ and $b$, must contain $E[a,b]$.

ii-) An arbitrary intersection of connected subsets of $X$ is either empty or connected.
\end{lemma}

\begin{proof}
The first part is a direct consequence of the fact that any point in the set $E[a,b]\setminus\{a,b\}$ separates $a$ and $b$. For the second part, let $\{M_\alpha\}$ be a collection of connected subsets of $X$ and suppose that $\cap_\alpha M_\alpha$ has at least two points. Given any pair of distinct points $a ,b \in \cap_\alpha M_\alpha$, we must have $E[a,b] \subset M_\alpha$ for all $\alpha$, whence we obtain $E[a,b] \subset \cap_\alpha M_\alpha$. The connectedness of $\cap_\alpha M_\alpha$ is now obvious.
\end{proof}

The above lemma will be used very often later.

\begin{lemma} \label{Lemma:Extension}
Let $D \subset X$ be connected and $x^* \in \overline{D} \setminus D$. Suppose that the function $F \in \HCX$ is nondegenerate at $x^*$ and consider $w \in C(D)$ such that $F(x,w(x)) = 0$ for all $x \in D$. Then, there exists the limit
\begin{equation}
\lim_{x \to x^*} w(x),\quad x \in D,
\end{equation}
in the Riemann sphere. 
\end{lemma}

\begin{proof}
Denote the Riemann sphere by $\widehat{\CC} = \CC \cup \{\infty\}$ and let $\{U_\alpha\}_{\alpha \in I}$ be a local basis at $x^*$ consisting of connected open sets. It is readily seen that the family $\mathcal{F} = \bigl\{\overline{w(D \cap U_\alpha)}\: :\:\alpha \in I\bigr\}$ is a filterbase in $\widehat{\CC}$. Since the latter is compact, $\mathcal{F}$ has at least one accumulation point, i.e.,
\begin{equation}
\mathcal{F}_{ac} = \bigcap_{\alpha \in I} \overline{w(D \cap U_\alpha)} \neq \emptyset.
\end{equation}

Next, by Lemma \ref{Lemma:E[a,b]}, it is easy to see that $D \cap U_\alpha$ is connected for all $\alpha \in I$ and the continuity of $w$ implies that $\overline{w(D \cap U_\alpha)}$ is also connected. Suppose that $\mathcal{F}_{ac}$ is not connected, i.e., there exist disjoint open sets $A, B \subset \widehat{\CC}$ such that $\mathcal{F}_{ac} \subset A \cup B$, $\mathcal{F}_{ac} \cap A \neq \emptyset$ and $\mathcal{F}_{ac} \cap B \neq \emptyset$. Note that we can write
\begin{equation}
\bigcap_{\alpha \in I} \overline{w(D \cap U_\alpha)}\ \cap \ \bigl(\widehat{\CC}\setminus(A \cup B)\bigr) = \mathcal{F}_{ac} \ \cap \ \bigl(\widehat{\CC}\setminus(A \cup B)\bigr) = \emptyset
\end{equation}
and accordingly, the compactness of $\widehat{\CC}$ implies the existence of a finite set of indices $\alpha_1, \ldots,\alpha_n \in I$ such that $\overline{w(D \cap U_{\alpha_1})} \cap \ldots \cap \overline{w(D \cap U_{\alpha_n})} \cap \bigl(\widehat{\CC}\setminus(A \cup B)\bigr) = \emptyset$. Since $\mathcal{F}$ is a filterbase, we can find $\beta \in I$ such that $\overline{w(D \cap U_\beta)} \subset \overline{w(D \cap U_{\alpha_1})} \cap \ldots \cap \overline{w(D \cap U_{\alpha_n})}$ and thus, $\overline{w(D \cap U_\beta)} \subset A \cup B$. However, as $\mathcal{F}_{ac} \subset \overline{w(D \cap U_\beta)}$, we must have $\overline{w(D \cap U_\beta)} \cap A \neq \emptyset$ and $\overline{w(D \cap U_\beta)} \cap B \neq \emptyset$. Hence, $\overline{w(D \cap U_\beta)}$ cannot be connected, which is absurd.

We assume, towards contradiction that $\mathcal{F}_{ac}$ contains at least two points. Let $\epsilon>0$ be arbitrary and let $z^* \in \mathcal{F}_{ac}$, $z^* \neq \infty$. Pick $\delta > 0$ and a neighborhood $U_\gamma$ of $x^*$ with $\gamma\in I$ such that $|F(x,z) - F(x^*,z^*)|<\epsilon$ whenever $x \in U_\gamma$ and $|z-z^*|<\delta$. Since $z^* \in \overline{w(D \cap U_\gamma)}$, there exists $x_\gamma \in D \cap U_\gamma$ such that $|w(x_\gamma) - z^*|<\delta$, whence we obtain that $|F(x_\gamma,w(x_\gamma)) - F(x^*,z^*)|<\epsilon$. Given that $F(x_\gamma,w(x_\gamma)) = 0$, we must have $|F(x^*,z^*)| < \epsilon$. Since $\epsilon$ is arbitrary, $F(x^*,z^*) =0$. Therefore, any finite point of $\mathcal{F}_{ac}$ is a root of $z \mapsto F(x^*,z)$. Since $F$ is nondegenerate at $x^*$, $z \mapsto F(x^*,z)$ is a non-constant entire function and therefore has at most countably many roots. As a result, $\mathcal{F}_{ac}$ is at most countable. Since it is also a non-empty,
connected subset of $\widehat{\CC}$, we get our desired contradiction and conclude that $\mathcal{F}_{ac}$ reduces to a single point. Then, it is straightforward to see that such point must be the limit of $w(x)$ as $x \to x^*$. 
\end{proof}

We now prove the main result of the paper.

\begin{theorem} \label{Thm:Main}
Let $F \in \HCX$ be  a nowhere degenerate function, having no asymptotic roots and assume that there exist $x_0 \in X$ and $z_0 \in \CC$ such that $F(x_0, z_0) = 0$. Then there exists $w \in C(X)$ such that $w(x_0) = z_0$ and $F(x,w(x))=0$ for all $x \in X$.
\end{theorem}

\begin{proof}
Let $\DD$ be the set of pairs $(D,w)$, where $D \subset X$ is a connected subset containing $x_0$, $w \in C(D)$, $w(x_0) = z_0$ and $F(x,w(x))=0$ for all $x \in X$. The family $\DD$ is not empty, as it contains the pair $(D_0,w_0)$, where $D_0=\{x_0\}$ and $w_0:D_0 \to \CC$ is defined by $w_0(x_0)=z_0$. We define a partial order in $\DD$ as follows: we write $(D_1,w_1) \leq (D_2,w_2)$ if $D_1 \subset D_2$ and $w_2|_{D_1} = w_1$.

Let $\{(D_\alpha,w_\alpha)\}_{\alpha \in I}$ be a chain in $\DD$. Set $\widetilde{D} = \bigcup_\alpha D_\alpha$ and define $\widetilde{w}:\widetilde{D} \to \CC$ by $\widetilde{w}(x) = w_\alpha(x)$, if $x \in D_\alpha$. It is obvious that $\widetilde{D}$ is a connected subset of $X$ containing $x_0$ and $\widetilde{w}$ is a well defined function such that $\widetilde{w}(x_0) = z_0$ and $F(x,\widetilde{w}(x))=0$ for all $x \in \widetilde{D}$. 

We subsequently prove that $\widetilde{w}$ is continuous on $\widetilde{D}$. Let $\widetilde{x} \in \widetilde{D}$ be arbitrary and consider a local basis $\{U_\beta\}_{\beta \in J}$ at $\widetilde{x}$ consisting of connected open sets. The family $\mathcal{F} = \bigl\{\overline{\widetilde{w}(\widetilde{D} \cap U_\beta)}\: :\:\beta \in J\bigr\}$ may be regarded as a filterbase in $\widehat{\CC}$. If we denote its set of accumulation points by $\mathcal{F}_{ac} = \bigcap_{\beta} \overline{\widetilde{w}(\widetilde{D} \cap U_\beta)}$, it is obvious that $\widetilde{w}(\widetilde{x}) \in \mathcal{F}_{ac}$, since $\widetilde{x} \in \widetilde{D} \cap U_\beta$ for all $\beta \in J$.

We show that $\widetilde{w}(\widetilde{D} \cap U_\beta)$ is connected for all $\beta \in J$. Suppose on the contrary that there exist two disjoint open sets $A, B \subset \widehat{\CC}$ such that $\widetilde{w}(\widetilde{D} \cap U_\beta) \subset A\cup B$, $\widetilde{w}(\widetilde{D} \cap U_\beta) \cap A \neq \emptyset$ and $\widetilde{w}(\widetilde{D} \cap U_\beta) \cap B \neq \emptyset$. Pick $\xi_A \in \widetilde{w}(\widetilde{D} \cap U_\beta) \cap A$ and $\xi_B \in \widetilde{w}(\widetilde{D} \cap U_\beta) \cap B$. Then, we can find $x_A, x_B \in \widetilde{D} \cap U_\beta$ such that $\widetilde{w}(x_A) = \xi_A$ and $\widetilde{w}(x_B) = \xi_B$. Note that
\begin{equation}
x_A \in \Biggl(\bigcup_{\alpha \in I} D_\alpha \Biggr) \cap U_\beta = \bigcup_{\alpha \in I} (D_\alpha \cap U_\beta)
\end{equation}
and accordingly, there exists an index $\alpha_1 \in I$ such that $x_A \in D_{\alpha_1} \cap U_\beta$. Similarly, there exists $\alpha_2 \in I$ such that $x_B \in D_{\alpha_2} \cap U_\beta$. Since $\{(D_\alpha,w_\alpha)\}_{\alpha \in I}$ is a chain, we may assume $D_{\alpha_1} \subset D_{\alpha_2}$. In that case, $x_A, x_B \in D_{\alpha_2} \cap U_\beta$, whence we derive that $E[x_A,x_B] \subset D_{\alpha_2} \cap U_\beta$, by an application of Lemma \ref{Lemma:E[a,b]}. Observe that $\widetilde{w}(E[x_A,x_B]) = w_{\alpha_2}(E[x_A,x_B])$ is connected; however, $\widetilde{w}(E[x_A,x_B]) \subset A \cup B$, $\xi_A \in \widetilde{w}(E[x_A,x_B]) \cap A$ and $\xi_B \in \widetilde{w}(E[x_A,x_B]) \cap B$, which is clearly impossible. We have reached a contradiction, which proves the connectedness of $\widetilde{w}(\widetilde{D} \cap U_\beta)$ for all $\beta \in J$. Therefore,  $\overline{\widetilde{w}(\widetilde{D} \cap U_\beta)}$ is also connected and an analogous argument to that of Lemma \ref{Lemma:Extension} shows that $\mathcal{F}_{ac}$ must be connected as well.

Also, by reviewing the techniques introduced in the proof of Lemma \ref{Lemma:Extension}, it is straightforward to see that any finite point of $\mathcal{F}_{ac}$ is a zero of the non-constant entire function $z \mapsto F(\widetilde{x},z)$, which shows that $\mathcal{F}_{ac}$ is at most countable. Since it is also non-empty and connected, it must reduce to a single point, which in this case is obviously $\widetilde{w}(\widetilde{x})$. Then, it is easy to conclude that $\widetilde{w}$ is continuous at $\widetilde{x}$.

A standard application of Zorn's Lemma shows that $\DD$ has a maximal element, which we denote by $(D^*,w^*)$. We wish to prove that $D^*=X$.

We first show that $D^*$ is closed. Conversely, suppose that there exists $x^* \in \overline{D^*}\setminus D^*$. A direct application of Lemma \ref{Lemma:Extension} shows that $w^*(x)$ has a limit in the Riemman sphere as $x \to x^*$ ($x \in D^*$), which cannot be infinity by the assumption on the non-existence of asymptotic roots for $F$. Therefore, $w^*$ has a continuous extension $\widetilde{w}^*$ to $D^* \cup \{x^*\}$. Note that the map $x \mapsto F(x,\widetilde{w}^*(x))$ vanishes on $D^*$ and is continuous on the connected set $D^* \cup \{x^*\}$, whence we deduce that $F(x,\widetilde{w}^*(x)) = 0$ for all $x \in D^* \cup \{x^*\}$. Consequently, we have proven that $(D^*,w^*) < (D^* \cup \{x^*\},\widetilde{w}^*)$, which contradicts the maximality of $(D^*,w^*)$.

Finally, suppose that $D^* \neq X$, i.e., there exists $y \in X\setminus D^*$. Since, as noted in page 4, $E[x_0,y]$ is order-complete with respect to the separation order, there exists a least upper bound $m$ of $E[x_0,y] \cap D^*$. Since $D^*$ is closed, it is easy to see that $m \in D^*$; moreover, we have the inclusions $E[x_0,m] \subset D^*$ (by Lemma \ref{Lemma:E[a,b]}) and $E[m,y]\setminus\{m\} \subset X \setminus D^*$. By taking into account that $F(m,w^*(m))=0$ and $F$ is nowhere degenerate, we can use Lemma \ref{Lemma:Algebrization} to find an open disk $D_r(w^*(m))$ and a neighborhood $V$ of $m$ such that $F(x,z)=P(x,z)\:G(x,z)$ for all $(x,z) \in V \times D_r(w^*(m))$, where $P$ is a monic polynomial with coefficients in $C(V)$ and $G$ is free of zeros in $V \times D_r(w^*(m))$. Without loss of generality, we may assume that $V$ is connected and then, we select $y_1 \in E[m,y]\setminus\{m\}$ such that $E[m,y_1] \subset V$. Since $E[m,y_1]$ is a totally ordered and order-complete space, we can find $w_1 \in C(E[m,y_1])$ such that $P(x,w_1(x))=0$ for all $x \in E[m,y_1]$, by \cite[Theorem 3]{DP1964}. Also, given that $P(m,z)$ is a power of $(z-w^*(m))$ (see Lemma \ref{Lemma:Algebrization}), we must have $w_1(m) = w^*(m)$. By the continuity of $w_1$, we can pick $\bar{y} \in E[m,y_1]\setminus\{m\}$ such that $w_1(E[m,\bar{y}]) \subset D_r(w^*(m))$. Now, we write $\widetilde{D} = D^* \cup E[m,\bar{y}]$ and consider the function $\widetilde{w}:\widetilde{D} \to \CC$ defined by
\begin{equation}
\widetilde{w}(x) = \begin{cases}
w^*(x), & x \in D^*;\\
w_1(x), & x \in E[m,\bar{y}].
\end{cases}
\end{equation}
It is easy to see that $D^*\setminus\{m\}$ and $E[m,\bar{y}]\setminus\{m\}$ are both open in $\widetilde{D}$, whence it may be inferred that $\widetilde{w}$ is continuous on $\widetilde{D}$. We prove that $F(x, \widetilde{w}(x))=0$ for all $x \in \widetilde{D}$. The result is obvious for $x \in D^*$. On the other hand, if $x \in E[m,\bar{y}]$, then it is straightforward to see that  $\widetilde{w}(x) \in D_r(w^*(m))$ (recall the choice of $\bar{y}$) and consequently, we have $F(x,\widetilde{w}(x)) = P(x,\widetilde{w}(x))\:G(x,\widetilde{w}(x)) = 0$. Thus, we have shown that $(D^*,w^*) < (\widetilde{D},\widetilde{w})$, which contradicts the maximality of $(D^*,w^*)$. The proof is now complete.
\end{proof}

\begin{remark}
Note that we have assumed that $X$ is connected in the preceding theorem, while Miura and Niijima \cite{MN2003} have shown that such restriction is unnecessary for $C(X)$ to be algebraically closed. Can we drop the connectedness hypothesis in Theorem \ref{Thm:Main}? Not completely. The connected components of a locally connected space are open. Hence, if we can find a root of $F$ in $C(X_\lambda)$ for every connected component $X_\lambda$ of $X$,  we easily conclude that $F$ has a root in $\CX$. If $F$ is nowhere degenerate and has no asymptotic roots, this can be done by Theorem \ref{Thm:Main}, \emph{provided that $F(x_0,z_0) = 0$ for some $x_0 \in X_\lambda$ and $z_0 \in \CC$}. Such condition is not always met for arbitrary functions $F \in \HCX$ (e.g., take $F$ to be a suitable exponential function in one connected component of $X$). However, if $F$ is a non-constant monic polynomial, it is trivially fulfilled and we may recover the results from \cite{MN2003}.
\end{remark} 

\begin{remark}
The restrictions imposed to $F$ in the hypotheses of Theorem \ref{Thm:Main} are not necessary for the existence of roots. For example, consider the algebra $C([0,1])$ and define $F_1(x,z) = \exp(xz) - 1$. It is clearly degenerate at $x_0=0$. Moreover, the function $\omega:(0,1] \to \CC$ defined by $\omega(x) = 2\pi i x^{-1}$ is an asymptotic root of $F_1$. However, it obviously has the zero function as a root.
\end{remark}

To finish this paper, we introduce two examples showing how the presence of degeneracy and asymptotic roots can interfere with the existence of roots.

\begin{example} Recall that $F$ is degenerate at $x_0 \in X$ if $z \mapsto F(x_0,z)$ is a constant map. Obviously, if it is not the zero map, $F$ cannot have any root. On the other hand, let $X=[0,1]$ and write $h(x)=\sin(1/x)$. Consider the function
\begin{equation}
F(x,z) = \begin{cases}
x\bigl(\exp z - \exp {h(x)}\bigr), & 0 < x \leq 1; \\
0, & x = 0.
\end{cases}
\end{equation}

It can be easily verified that $F\in \HCX$. Also, note that $F$ is degenerate at $x_0=0$ and $z \mapsto F(0,z)$ is the zero function. Suppose that $w \in \CX$ is a root of $F$. Then, $F(x,w(x))=0$ for all $x\in[0,1]$ implies that $w(x)=h(x)+2k(x)\pi i$ for $x \in (0,1]$, where $k(x) \in \mathbb{Z}$. By continuity, $k(x)$ must be constant, which yields $w(x)=\sin(1/x)+2k\pi i$ for all $x \in (0,1]$. Since this function does not have a continuous extension to the interval $[0,1]$, we have reached a contradiction. Moreover, although the function $g(x) = \sin(1/x)+2k\pi i$ satisfies $F(x,g(x))=0$ for all $x \in (0,1]$, it does not have a limit in the Riemann sphere as $x \to 0$. Therefore, the hypothesis of nondegeneracy is also essential for Lemma \ref{Lemma:Extension}. 
\end{example}

\begin{example}
Let $X=[0,1]$. Consider the function $\varphi(z)=z\exp(-z)$ and any continuous curve $\omega:[0,1)\to \CC$ such that $\omega(0)=0$, $\omega(x)=(1-x)^{-1}$ for $1/2\leq x<1$ and its image avoids the point $1$ (the zero of $\varphi '$). Define the function
\begin{equation}
F(x,z) = \begin{cases}
\varphi(z) - \varphi(\omega(x)), & 0 \leq x < 1; \\
\varphi(z), & x = 1.
\end{cases}
\end{equation}

It can be easily seen that $F \in \HCX$ and is nowhere degenerate; however, $\omega$ is an asymptotic root of $F$. Suppose that $w\in \CX$ is a root of $F$. Then, we must have $\varphi(w(x)) = \varphi(\omega(x))$ for all $x \in [0,1)$. We prove that the set $A = \{x\in [0,1)\:|\: w(x)=\omega(x)\}$ is open and closed in $[0,1)$. The second assertion is obvious from the continuity of $w-\omega$. On the other hand, if $w(x_0) = \omega(x_0) = z_0$, we have that $\varphi$ is locally injective at $z_0$ (since $\varphi'(\omega(x)) \neq 0$ for all $x \in [0,1)$). Since $\varphi(w(x)) = \varphi(\omega(x))$, the continuity of $w$ and $\omega$ implies that such functions must coincide in a neighborhood of $x_0$, proving that $A$ is open in $[0,1)$. Next, note that $0\in A$. Since $[0,1)$ is connected, we conclude that $A=[0,1)$. However, this means that $w(x) \to \infty$ as $x \to 1$, which is clearly absurd.
\end{example}

\bibliographystyle{amsplain}

\end{document}